\documentclass[12pt]{article}

\usepackage[a4paper, left=3.5cm,top=3.1cm,right=3.5cm,bottom=3.5cm]{geometry}
\usepackage{concmath}
\usepackage[T1]{fontenc}

\usepackage{graphicx}
\usepackage{mathtools}
\usepackage{psfrag}
\usepackage{amssymb}
\usepackage{amsmath}
\usepackage{authblk}
\usepackage{siunitx}
\usepackage{bm}

\usepackage[utf8]{inputenc}

\usepackage{amsthm}
\usepackage[hyphens]{url}
\usepackage{hyperref}\usepackage{todonotes}

\usepackage{algorithm}
\usepackage{algorithmic}
\usepackage{subfig}

\usepackage{xspace}

\usepackage{enumitem}
\setitemize{label=\scriptsize{$\blacksquare$},topsep=0em}
\setenumerate{label=(\alph*), topsep=0em}

\newtheorem{theorem}{Theorem}
\newtheorem{lemma}[theorem]{Lemma}

\usepackage{todonotes}

\DeclarePairedDelimiter{\innerprod}{\langle}{\rangle}

\usepackage[hyphens]{url}
\usepackage{hyperref}\usepackage{todonotes}

\usepackage{algorithm}
\usepackage{algorithmic}
\usepackage{subfig}

\usepackage{xspace}

\usepackage{enumitem}
\setitemize{label=\scriptsize{$\blacksquare$}}
\setenumerate{label=(\alph*)}

\newtheorem{definition}[theorem]{Definition}

\newcommand{\R}{\mathbb R}

\newcommand{\N}{\mathbb N}

\newcommand{\radon}{\mathcal{R}}
\newcommand{\xx}{c}

\newcommand{\yy}{y}
\newcommand{\Ao}{\mathbf A}                         

\newcommand{\uu}{u}
\newcommand{\vvv}{v}

\def\plus{{\boldsymbol{\texttt{+}}}}

\newcommand{\al}{\alpha}

\newcommand{\ran}{\operatorname{ran}}
\newcommand{\dom}{\operatorname{dom}}

\newcommand\abs[1]{\left\vert#1\right\vert}

\newcommand\norm[1]{{\left\Vert#1\right\Vert}}
\newcommand\snorm[1]{\Vert#1\Vert}

\makeatletter
\newcommand*\bigcdot{\mathpalette\bigcdot@{.6}}
\newcommand*\bigcdot@[2]{\mathbin{\vcenter{\hbox{\scalebox{#2}{$\m@th#1\bullet$}}}}}
\makeatother

\newcommand{\signal}{x}

\newcommand{\data}{y}
\newcommand{\ndata}{y_{\delta}}

\newcommand{\Po}{\mathbf  P}                         

\newcommand{\Bo}{\mathbf  B}                         
\newcommand{\Ro}{\mathbf R}    

\newcommand{\M}{\mathbb M}                         
\newcommand{\X}{\mathbb X}                         
\newcommand{\Y}{\mathbb Y}

\newcommand{\Id}{\operatorname{Id}}

\newcommand{\NN}{\mathbf{U}}  
\newcommand{\nun}{\mathbf{N}}  
\newcommand{\nullnet}{\mathbf{N}}

\usepackage{xcolor}
\colorlet{lred}{red!40}
\colorlet{lgreen}{green!40}
\colorlet{lblue}{blue!40}

\numberwithin{equation}{section}
\numberwithin{figure}{section}
\numberwithin{theorem}{section}

\allowdisplaybreaks

\title{Big in Japan: Regularizing networks for solving inverse problems}

\author{Johannes~Schwab}
\author{Stephan~Antholzer}
\author{Markus~Halti~Haltmeier}

\affil{Department of Mathematics, University of Innsbruck\authorcr
Technikerstrasse 13, 6020 Innsbruck, Austria
\authorcr
{\tt johannes.schwab@uibk.ac.at},
{\tt stephan.antolzer@uibk.ac.at},
{\tt  markus.haltmeier@uibk.ac.at}
}

\date{September 14, 2019}

\begin{document}

\maketitle

\begin{abstract}
Deep learning and (deep) neural networks are emerging tools to address inverse problems
and image reconstruction tasks. Despite outstanding performance,  the mathematical analysis
for solving inverse problems by neural networks  is mostly missing. In this paper, we introduce
and rigorously analyze families of deep regularizing neural networks (RegNets) of the form $\Bo_\al + \nun_{\theta(\al)}  \Bo_\al$, where $\Bo_\al$ is a classical regularization and the network $\nun_{\theta(\al)}  \Bo_\al $
is trained to recover the missing part  $\Id_X - \Bo_\al$ not found by the classical regularization.   We show that these regularizing  networks yield a convergent regularization method for solving inverse problems. Additionally, we derive convergence rates (quantitative error estimates) assuming a sufficient decay of the associated distance function.
We demonstrate that our results recover existing convergence and convergence rates  results for filter-based regularization methods
as well as the recently introduced null space network as  special cases.  Numerical results are presented for a
tomographic sparse data problem, which  clearly demonstrate that  the proposed RegNets  improve the classical regularization as well as the  null space network.


\end{abstract}

\section{Introduction}
\label{sec:intro}

This paper is concerned with solving inverse problems of the form
\begin{equation}\label{eq:ip}
\ndata  =  \Ao\signal+ z \,,
\end{equation}
where $\Ao\colon \X\rightarrow \Y$ is a bounded linear operator between Hilbert spaces $\X$ and  $\Y$, and $z$ denotes the data distortion that satisfies $\norm{z} \leq \delta$ for some noise level $\delta\geq 0$. Many inverse problems arising in medical imaging, signal processing, astronomy, computer vision and other fields can be written in the form \eqref{eq:ip}. A main characteristic property of  inverse problems is that they are ill-posed \cite{engl1996regularization,scherzer2009variational}. This means that the solution of \eqref{eq:ip} is either not unique or is unstable with respect to  data perturbations.

To solve such kind of inverse problems one has to employ regularization methods,
which serve the following two main purposes:
\begin{itemize}
\item Select particular solutions of the noise-free equation, thereby accounting for non-uniqueness $\ker(\Ao ) \neq \{0 \}$.
\item Approximate \eqref{eq:ip} by neighboring but  stabler problems.
\end{itemize}
Our aim is  finding convergent regularization methods for the solution of \eqref{eq:ip} using  deep neural networks  that can be adjusted to
realistic training data.

In \cite{schwab2018deep} we focused on the non-uniqueness issue, where particular solutions of the noise-free equation, \eqref{eq:ip} with $z=0$, are approximated using classical regularization methods combined with null space networks. Null space networks (introduced originally in \cite{mardani2017deep} in a finite dimensional setting)
are refined residual networks, where the residual is projected onto the null space of the operator $\Ao$.
In this context, the stabilization of finding a solution to \eqref{eq:ip} comes from a given traditional regularization method and the role of the network is to select  correct solutions in a data consistent  manner.

\subsection*{Proposed regularizing networks (RegNets)}

In this paper, we go one step further and generalize the concept of deep null space learning
by allowing the  network to also act in the orthogonal complement of the null space of $\Ao$ in a controlled manner. This is in particular useful if the operator contains several small singular values
that are not strictly equal to zero. Similar to the components in the kernel, these parts are difficult to be reconstructed by a classical linear regularization method and quantitative
error estimates require strong smoothness assumptions on the objects to be recovered. Learning almost invisible components can significantly improve reconstruction results for less smooth objects.

The proposed RegNets generalize the structure of null space networks
analyzed in \cite{schwab2018deep}  and consist of a
 family $(\Ro_\al)_{\al>0}$ of mappings  $\Ro_\al \colon \Y \rightarrow \X$  of the form
\begin{equation}\label{def:regnet}
\Ro_\al
\coloneqq
\Bo_\al +  \nun_{\theta(\al)}  \Bo_\al
\quad \text{ for } \al >0 \,.
\end{equation}
Here $(\Bo_\al)_{\al >0}$ with  $\Bo_\al\colon \Y\rightarrow \X$
is a classical regularization of the Moore-Penrose inverse $\Ao^\plus$, and $
\nun_{\theta(\al)}   \colon \X \to \X
$ are  neural networks that can be trained to map the part $\Bo_\al \Ao \signal$ recovered by  the regularization method to the missing part $(\Id_X-\Bo_\al \Ao) \signal$. Here  $(\nun_{\theta})_{\theta \in \Theta}$ is any family of parameterized functions
that  can be taken as a standard network, for example a convolutional neural network (CNN). In particular,  $\nun_{\theta(\al)}  $ 
 is allowed to depend on the regularization parameter $\alpha$.  

In this paper we  show that if  $\nun_{\theta(\al)}  \Bo_\al \Ao  \rightarrow \nullnet $ on  $\ran(\Ao^\plus)$   as $\al\rightarrow 0$  for some function $ \nullnet \colon \X\rightarrow \X$ with $\ran( \nullnet)\subseteq \ker(\Ao)$, the RegNets   defined by \eqref{def:regnet}  yield a convergent  regularization method with admissible set $
\M \coloneqq (\Id_X + \nullnet) ( \ran(\Ao^\plus) ) $. Further we derive  convergence rates
(quantitative error estimates) for elements satisfying  conditions different from the classical smoothness assumptions.

\subsection*{Outline}

The organization of this paper is as follows.
In Section \ref{sec:regnet} we present some background and related results.
In Section~\ref{sec:regprops} we introduce  the proposed
regularizing networks and  show that they yield a convergent
regularization method. Further, we derive convergence rates  under a modified source condition. In Section~\ref{sec:spcases} we demonstrate that our results contain existing
convergence results as special cases. This includes filter-based methods, classical Tikhonov regularization, and regularization by null space networks.
Moreover, we examine a data driven extension of singular components,
where the classical regularization method is given by truncated singular value decomposition (SVD). The paper concludes with a short  summary  presented in Section~\ref{sec:conclusion}.

\section{Some background}

\label{sec:regnet}

Before actually analyzing the  RegNets, we
recall basic notions and concepts from
regularization  of inverse problems (see \cite{scherzer2009variational,engl1996regularization}) 
 and the concept
of null space networks. We also review some previous
related work.

\subsection{Classical regularization of inverse problems}

Regularization methods to stably find a solution of \eqref{eq:ip} use a-priori information about the unknown, for example that the solution $\signal$ lies in a particular set of admissible elements $\M$. For such a set $\M\subseteq \X$, a regularization method is a tuple $((\Bo_\al)_{\al >0},\al^\star)$, where  $\Bo_\al\colon \Y\rightarrow \X$ are  continuous operators and  $\al^\star(\delta,\ndata)$ is a  parameter choice function such that for all $x\in \M$ we have $\Bo_{\al^{\star}(\delta,\ndata)}(\ndata)\rightarrow x$ as $\delta \rightarrow 0$.

Classical regularization methods approximate the Moore-Penrose inverse $\Ao^\plus$ and the set $\M$ is given by $\M =\ker(\Ao)^\perp$. Note that for any $y \in \ran(\Ao)$, the Moore-Penrose inverse $\Ao^\plus y$ is given by the minimal norm solution of \eqref{eq:ip}. A precise  definition of a regularization
method is as follows.

\begin{definition}[Regularization method]\label{def:regmeth}
Let $(\Bo_\al)_{\al>0}$ a family of continuous operators $\Bo_\al\colon \Y\rightarrow \X$ and suppose $\al^\star\colon (0,\infty)\times \Y \rightarrow (0,\infty)$. The pair $((\Bo_\al)_{\al>0},\al^\star)$ is called a (classical) regularization method for the solution of $\Ao\signal=\data$ with $\data\in \dom(\Ao^\plus)$, if the following holds
\begin{itemize}
\item $\lim_{\delta\rightarrow 0} \sup\{\al^\star(\delta,\ndata)\mid \ndata\in \Y\,,  \|\ndata-y\|\leq\delta\}=0$.
\item $\lim_{\delta\rightarrow 0} \sup\{\|\Ao^\plus\data-\Bo_{\al^\star(\delta,\ndata)}\ndata \| \mid \ndata\in \Y \text{ and } \|\ndata-y\|\leq\delta\}=0$.
\end{itemize}
\end{definition}

The parameter choice $\al^\star$, depending on the noise level as well as on the data, determines the level of approximation of the Moore-Penrose inverse. For decreasing noise level the ill-posed problem \eqref{eq:ip} can be approximated by stable problems getting closer to finding the minimum norm solution of \eqref{eq:ip} and in the limit it holds $\lim_{\delta\rightarrow 0}\Bo_{\al^\star(\delta,\ndata)}(\ndata)=\Ao^\plus \data$.

A great variety of regularization methods, namely filter-based regularization methods,
can be defined by regularizing filters.

\begin{definition}[Regularizing filter]\label{def:filtreg}
A family $(g_\al)_{\al>0}$ of piecewise continuous functions $g_\al\colon[0,\|\Ao^*\Ao\|]\rightarrow \R$ is called regularizing filter if
\begin{itemize}
\item $\sup\{|\lambda g_\al(\lambda)| \mid \al>0 \text{ and } \lambda \in [0,\|\Ao^* \Ao\|]\}< \infty$.
\item $\forall \lambda \in (0,\|\Ao^* \Ao\|]\colon \lim_{\al\rightarrow 0} g_\al(\lambda)=1/\lambda$.
\end{itemize}
\end{definition}

Any regularizing filter $(g_\al)_{\al>0}$  defines  a regularization method  by taking
\begin{equation}\label{eq:filterbased}
\forall \al >0 \colon \quad \Bo_\al\coloneqq g_\al(\Ao^* \Ao)\Ao^*\,.
\end{equation}
We call a regularization according to \eqref{eq:filterbased} a (classical) filter based regularization. Note that $\Ao^* \Ao \colon \X \to \X$ is a self-adjoint bounded linear operator, and therefore $g_\al(\Ao^* \Ao) \colon  \X \to \X$   is bounded linear as well, defined by  the framework of functional calculus \cite{Hel69,Wei80}. In particular, if $\Ao^* \Ao$ has an eigenvalue decomposition $\Ao^* \Ao (\signal) = \sum_{n\in\N} \lambda_n \langle u_n,\signal\rangle u_n$, then
\begin{equation*}
	\forall x \in \X \colon \quad
g_\al(\Ao^* \Ao) \signal \coloneqq  \sum_{n\in\N} g_\al(\lambda_n) \langle u_n, \signal\rangle u_n \,.
\end{equation*}
In the general case, the spectral decomposition of $\Ao^* \Ao$ is used to rigorously define $g_\al(\Ao^* \Ao) $, see \cite{Hel69,Wei80}. 

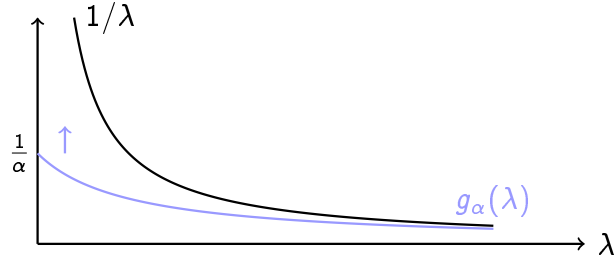
\begin{figure}[htb!]
\centering
\begin{tikzpicture}[scale=1.2]
\draw[line width =0.3mm, ->](0,0)--(6,0) node[right]{$\lambda$};
\draw[line width =0.3mm, ->](0,0)--(0,2.5);
\draw[line width =0.3mm, domain=0.4:5, samples=200] plot (\x,{1/\x});
\draw[line width =0.3mm, lblue, domain=0:5, samples=200] plot (\x,{1/(1+\x)})node[above]{$g_\alpha(\lambda)$};
\node at (-0.2,1){$\frac{1}{\alpha}$};
\draw[line width =0.3mm, lblue,->](0.3,1)--(0.3,1.3);
\node at (0.8,2.5){$1/\lambda$};
\end{tikzpicture}
\caption{\label{fig:filter1} Illustration of the regularizing filter for Tikhonov regularization.}
\end{figure}

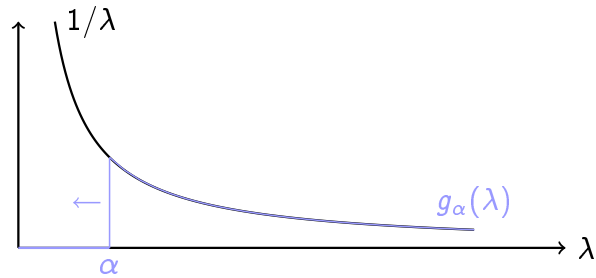
\begin{figure}[htb!]
\centering
\begin{tikzpicture}[scale=1.2]
\draw[line width =0.3mm, ->](0,0)--(6,0) node[right]{$\lambda$};
\draw[line width =0.3mm, ->](0,0)--(0,2.5);
\draw[line width =0.3mm, domain=0.4:5, samples=200] plot (\x,{1/\x});
\draw[line width =0.3mm, lblue, domain=0:1, samples=200] plot (\x,{0})node[below] {$\alpha$};
\draw[line width =0.2mm, lblue, domain=0:1] plot (1,{\x});
\draw[line width =0.2mm, lblue, domain=1:5] plot (\x,{1/\x})node[above]{$g_\alpha(\lambda)$};
\draw[line width =0.2mm, lblue,->](0.9,0.5)--(0.6,0.5);
\node at (0.8,2.5){$1/\lambda$};
\node at (0.8,2.5){\phantom{$1/\lambda$}};
\end{tikzpicture}
\caption{Illustration of the  regularizing filter \label{fig:filter2} for truncated SVD.}
\end{figure}

Two prominent examples of filter-based regularization methods are classical Tikhonov regularization and
truncated SVD.  In Tikhonov regularization, the regularizing filter is given by  $g_\al(\lambda)=1/(\lambda+\al)$,
see Figure~\ref{fig:filter1}. This yields $\Bo_\al=(\Ao^\ast \Ao +\al \Id_X)^{-1}\Ao^\ast$.
In truncated SVD,  the regularizing  filter is given by
\begin{equation}\label{eq:SVDfil}
g_\al(\lambda)=\begin{cases}
0, \quad &\lambda < \al \\
\frac{1}{\lambda} &\lambda \geq \al \,,
\end{cases}	
\end{equation}
see Figure \ref{fig:filter2}.
For both  methods the admissible set is $\M=\ker(\Ao)^\perp$.

Other typical filter-based regularization methods are the Landweber iteration and iterative
Tikhonov regularization~\cite{engl1996regularization}.

\subsection{Null space networks}

Standard regularization approximates the Moore Penrose inverse and therefore selects elements in $\ker(\Ao)^\perp$.
In \cite{schwab2018deep} we introduced regularization of  null space networks, where the aim is to approximate elements in a set  $\M$ different from $\ker(\Ao)^\perp$.

Null space networks are defined as follows.

\begin{definition}[Null space network]\label{def:nsn}
We call a function
$\Id_X + \nullnet \colon \X \to \X$
 a  null space network  if $ \nullnet =\Po_{\ker(\Ao)}  \NN$
  where $\NN \colon \X \to \X$ is  any Lipschitz  continuous function.
\end{definition}

Moreover   we use the following generalized notion of a
regularization method.

\begin{definition}[Regularization methods with admissible set $\M$]\label{def:regmeth1}
Let $(\Ro_\al)_{\al>0}$ be a family of continuous operators $\Ro_\al\colon \Y\rightarrow \X$ and $\al^\star\colon (0,\infty)\times \Y \rightarrow (0,\infty)$. Then the pair $((\Ro_\al)_{\al>0},\al^\star)$ is  called a regularization method (for the solution of $\Ao\signal=\data$) with admissible set $\M$, if for all
$x\in\M$,  it holds
\begin{itemize}
\item $\lim_{\delta\rightarrow 0} \sup\{\al^\star(\delta,\ndata)\mid \ndata\in \Y \,,  \|\ndata- \Ao x \|\leq\delta\}=0$.
\item $\lim_{\delta\rightarrow 0} \sup\{\| x -\Ro_{\al^\star(\delta,\ndata)}\ndata \| \mid \ndata\in \Y \text{ and } \|\ndata-\Ao x\|\leq\delta\}=0$.
\end{itemize}
In this case we call $(\Ro_\al)_{\al>0}$ an $(\Ao,\M)$-regularization.
\end{definition}

\begin{figure*}[htb!]
\begin{tikzpicture}[scale=1]
\draw[line width=1.3pt,->] (0,0)--(7,-3);
\node at (8.2,-3.4){$\ran(\Ao^{\plus}) = \ker(\Ao)^\bot$};
\draw[line width=1.3pt,->] (3,-2)--(4.5,1.5);
\node at (4.5,1.8){$\ker(\Ao)$};
\draw [line width =0.6mm, lblue] plot [smooth] coordinates {(0.5,7/6) (3,1.4) (4,0.5) (5,1) (7,0.7) (8,1)};
\node at (0.5,8.3/6)[above] {$\M \coloneqq (\Id_X + \nullnet)(\ran(\Ao^{\plus}))$};
\filldraw (5,-15/7)circle(1.7pt);
\node at (4.4,-31/14)[below]{$\Bo_\al \ndata$};
\draw [dashed, line width=1.3pt,->] (5,-15/7)--(5+2.9*3/7,-15/7+2.9*1);
\filldraw[color=black] (5+2.95*3/7,-15/7+2.95*1)circle(3pt);
\node at (5+10/7,-15/7+3.1*1)[above]{$\Ro_\al \ndata$};
\node at (-3,0){};
\end{tikzpicture}
\caption{\label{fig:null} Regularization defined by
a  null space network.
For a filter-based regularization method we have
$\Bo_\al \ndata \in \ker(\Ao)^\bot$. The  regularized
null space network $\Ro_\al = \Bo_\al + \nullnet \circ \Bo_\al $  adds  reasonable parts along the null space $\ker(\Ao)$  to
the standard regularization $\Bo_\al \ndata $.}
\end{figure*}

The regularized null space networks analyzed in  \cite{schwab2018deep} take the form
\begin{equation}\label{eq:null-space}
\Ro_\al\coloneqq (\Id_X+\nullnet)\circ \Bo_\al \quad \text{for } \al>0 \,,
\end{equation}
where   $(\Bo_\al)_{\al>0}$ is any classical regularization method
and $\Id_X+\nullnet$ any null space network  (for example, defined by a  trained deep neural network).
In  \cite{schwab2018deep} we have shown that  \eqref{eq:null-space} yields a regularization method with admissible set
$ \M \coloneqq (\Id_X + \nullnet )(\ran(\Ao^\plus))$.
This approach is designed to find the null space component of the solution in a data driven manner with a fixed neural network $\nullnet$ independent of the regularization parameter $\al$, that works in the null space of $\Ao$; compare  Figure~\ref{fig:null}.

In this paper we go one  step further and
consider  a sequences of regularizing networks (RegNets)
of the form $(\Id_X + \nun_{\theta(\al)} )\circ \Bo_\al$ generalizing null space networks of the form \eqref{eq:null-space}.
Here $\nun_{\theta(\al)} $ depends on $\al$ and is allowed to act in the orthogonal complement of the kernel  $\ker(\Ao)^\perp$. We give conditions under which
 this approach yields   a regularization method with admissible set $\M$.

Allowing the  network $\nun_{\theta(\al)} $ to  also act  in $\ker(\Ao)^\perp$ in particular is beneficial, if the forward operator $\Ao$ contains  many  small  singular values.  In this case, the  network can  learn  components which are not
sufficiently well contained  in the data. Note that in the limit $\al \rightarrow 0$,  the regularization method $(\Bo_\al)_{\al >0}$ converges to $\Ao^\plus$ point-wise. Therefore, in the limit  $\al \to 0$, the network is  restricted to learn  components in the null space of $\Ao$.

\subsection{Related work}

Recently,  many works using deep neural networks to solve inverse problems have been published. These papers include two stage approaches, where in a first step an initial reconstruction is done, followed by a deep neural network. Several network architectures,
often  based on the U-net architecture \cite{ronneberger2015unet} and improvements of it \cite{ye2018deep,han2018framing}, have been used for this class of methods.

CNN based methods that only modify  the part of the reconstruction that is contained in  the null space of the forward operator have been proposed in \cite{mardani2017recurrent,mardani2017deep}. In \cite{schwab2018deep} we introduced regularized null space networks which   are  shown to lead a convergent regularization method. Recently, a related synthesis approach for learning the invisible frame coefficients for limited angle computed tomography has been proposed in \cite{bubba2019learning}.

Another possibility to improve reconstructions by deep learning is to replace certain operations in an iterative scheme by deep neural networks or use learned regularization functionals \cite{kobler2017variational,gupta2018cnn,li2018nett,adler2018banach,adler2017solving}. Further, a Bayesian framework has been proposed in \cite{adler2017learning,adler2018deep}, where the posterior distribution of solutions is approximated by learned CNNs.

\section{Convergence and convergence rates of RegNets}

\label{sec:regprops}

In this section, we formally introduce the concept of RegNets, analyze their regularization properties and derive convergence rates.

Throughout the following, let $\Ao\colon \X\rightarrow \Y$ be a linear and bounded operator and $\Id_X + \nullnet \colon \X\rightarrow \X$ be a null space network, see Definition~\ref{def:nsn}.
Further, let $(\Bo_\al)_{\al>0}$ denote a classical filter-based
regularization method, defined by the regularizing filter
 $(g_\al)_{\al>0}$, see Definition~\ref{def:filtreg}.

\subsection{Convergence}
\label{subsec:conv}

Let us first  formally define a
 family of regularizing networks.

\begin{definition}\label{def:regnets}
Let $(\Bo_\al)_{\al>0}$ be a classical filter-based regularization method.
A family $(\nun_{\theta(\al)} )_{\al>0}$  of Lipschitz continuous functions $\nun_{\theta(\al)} \colon \X\rightarrow \X$ is called\\
$((\Bo_\al)_{\al>0}, \nullnet)$-adapted if
\begin{itemize}
\item $ \lim_{\al\rightarrow 0} \nun_{\theta(\al)} (\Bo_\al \Ao z) = \nullnet(z)$
for all $z\in \ran(\Ao^\plus)$.
\item
The Lipschitz constants of $ (\nun_{\theta(\al)})_{\al >0} $
are bounded from above by some constant $L>0$.
\end{itemize}
\end{definition}

For the following recall Definition~\ref{def:regmeth1} of a regularization method  with admissible set $\M$. We will often  use the notation $ \nun z \coloneqq \nun(z)$. The  following convergence results hold.

\begin{theorem}[RegNets]\label{theo:conv}
Let $(\Bo_\al)_{\al>0}$ be a classical filter-based regularization method
and $(\nun_{\theta(\al)} )_{\al>0}$  be
$((\Bo_\al)_{\al>0}, \nullnet)$-adapted.
 Then  the  family
\begin{equation}\label{eq:reg}
\Ro_\al(\ndata)=(\Id_X+\nun_{\theta(\al)} )\Bo_\al(\ndata),
\end{equation}
is a regularization method with admissible set 
\begin{equation}\label{eq:Mset}
    \M \coloneqq  (\Id_X + \nullnet) ( \ran(\Ao^\plus) )\,.
\end{equation}
We call $(\Ro_\al)_{\al >0}$
a regularizing family of networks (RegNets) adapted
to $((\Bo_\al)_{\al>0}, \nullnet)$.
\end{theorem}

\begin{proof}
Let $\signal_{\al,\delta}\coloneqq \Ro_\al(\ndata)=(\Id_X+\nun_{\theta(\al)} )\Bo_\al(\ndata)$.
Then we have
\begin{align} \nonumber
\|\signal -&\signal_{\al,\delta}\|
 \\ \nonumber =&\|\Bo_\al \Ao \signal +(\Id_X-\Bo_\al \Ao)\signal
 -\Bo_\al \ndata-\nun_{\theta(\al)} \Bo_\al \ndata\|
 \\ \nonumber \leq \ &\|\Bo_\al(\Ao \signal -\ndata)\|
+\|(\Id_X-\Bo_\al \Ao)\signal -\nun_{\theta(\al)}  \Bo_\al \Ao \signal\|
\\ \nonumber &
+\|\nun_{\theta(\al)} \Bo_\al \Ao \signal-\nun_{\theta(\al)} \Bo_\al \ndata\|
\\ 
\leq  \ &(1+L)\|\Bo_\al\|\delta
+\|x-\nun_{\theta(\al)} \Bo_\al\Ao \signal -\Bo_\al \Ao \signal\| \,.\label{eq:ineq}
\end{align}
Assuming that $x=(\Id_X+\nullnet)z\in \M$ with $z\in\ran(\Ao^\plus)$ we get
\begin{align*}
\|\signal-& \signal_{\al,\delta}\|
\\ \leq  &(1+L)\|\Bo_\al\|\delta
+ \|z
 +\nullnet z-\nun_{\theta(\al)} \Bo_\al \Ao z-\Bo_\al \Ao z\| \\
\leq  &(1+L)\|\Bo_\al\|\delta + \|z-\Bo_\al \Ao z\|
+ \| \nullnet z-\nun_{\theta(\al)}  \Bo_\al \Ao z\|.
\end{align*}
Eventually we get $\lim_{\delta\rightarrow 0} \|\signal-\signal_{\al,\delta}\| = 0$ since the first expression vanishes by assumption, the second because $(\Bo_\al)_{\al>0}$ is a regularization method and the last because of $(\nun_{\theta(\al)} )_{\al>0}$ being $((\Bo_\al)_{\al>0}, \nullnet)$-adapted.
\end{proof}

\subsection{Convergence rates}

In this section, we derive convergence rates for RegNets introduced
 in Section~\ref{subsec:conv}.
  To that end, we first introduce a distance function and define the qualification of a  classical regularization method.
The definition of the distance function is essentially motivated by  \cite{hofmann2005convergence}.

\begin{definition}[Distance function]
For  any numbers $\al, \rho, \mu >0$ and $x \in \X$ we define the distance function
\begin{multline}\label{eq:dist}
d_\al(x; \rho, \mu)\coloneqq \inf\{\|x-\nun_{\theta(\al)}  \Bo_\al \Ao \signal -(\Ao^\ast \Ao)^\mu \omega\| \\ \mid \omega\in \X \wedge \|\omega\|\leq \rho\}.
\end{multline}
\end{definition}

The qualification of a regularization method is a classical concept in regularization
theory  (see  \cite[Theorem~4.3]{engl1996regularization}) and central for the
derivation of  convergence rates.

\begin{definition}[Qualification]\label{def:order}
We say that a filter based regularization $\Bo_\al\coloneqq g_\al(\Ao^\ast \Ao)\Ao^\ast$  defined by the regularizing filter $(g_\al)_{\al>0}$ has qualification at last $\mu_0 \in  (0,\infty)$
if there is a constant $C>0$ such that
for all $\mu \in (0,\mu_0]$  we have
\begin{equation} \label{eq:quali}
    \forall \al >0 \colon \sup\{\lambda^\mu \abs{ 1-\lambda g_\alpha(\lambda) } \mid  \lambda \in [0,\lVert \Ao^\ast\Ao \rVert]\} \le C \al^\mu \,.
\end{equation}
The largest value $\mu_0$  such that \eqref{eq:quali}
holds  for all
$\mu \in (0,\mu_0]$
is called the qualification of the regularization method
$(\Bo_\al)_{\al>0}$ or the regularizing filter $(g_\al)_{\al>0}$
(taken as infinity if \eqref{eq:quali} holds for all $\mu > 0$).
\end{definition}

Note that Tikhonov regularization has qualification $\mu_0=1$, and 
truncated SVD regularization has infinite qualification.  Further, if $(\Bo_\al)_{\al>0}$
has qualification $\mu_0$, then (see \cite{engl1996regularization})
\begin{align} \label{eq:q1}
    &\norm{ (\Id_X-\Bo_\al \Ao)(\Ao^\ast \Ao)^\mu \omega }
     \leq C \rho \al^\mu
    \\  \label{eq:q12}
    &\| \Ao (\Id_X - \Bo_\al \Ao)(\Ao^\ast \Ao)^\mu \omega \|
     \leq C \rho \al^{\mu+1/2}
\end{align}
holds for $\mu \leq \mu_0$, $\al >0$ and all $\omega \in \X$
with $\norm{\omega} \leq \rho$.

\begin{lemma}\label{lem:est}
Let $(\Ro_\al)_{\al >0}$ be a family of RegNets adapted to
$((\Bo_\al)_{\al>0}, \nun)$ where  $(\Bo_\al)_{\al >0}$ has qualification of order at least   $\mu$.  Then, for any
$\al, \delta, \rho >0$ and $x \in \X$,
\begin{multline}
\|\Ro_\al(\ndata)-x\|
  \leq \delta (1+L)\|\Bo_\al\|
\\ \quad +C\rho \al^\mu+d_\al(x; \rho, \mu)
+ \|\Bo_\al\Ao \nun_{\theta(\al)}  \Bo_\al \Ao \signal\|  \,,\label{eq:est}
\end{multline}
where $\ndata \in \Y$ satisfies $\snorm{\Ao x -\ndata }\leq   \delta$ and $C$ is the  constant from Definition \ref{def:order}.
\end{lemma}

\begin{proof}
As in the proof of Theorem \ref{theo:conv} we have
\begin{equation}
\|\signal-\signal_{\al,\delta}\|\leq (1+L)\|\Bo_\al\|\delta  + \underbrace{\|\signal-\nun_{\theta(\al)} \Bo_\al \Ao \signal-\Bo_\al \Ao \signal\|}_{\eqqcolon E_\al}.
\end{equation}
Further for all $\omega\in \X$ with $\|\omega\|\leq \rho$, the term $E_\al$ can be estimated as
\begin{align*}
E_\al \leq  &\norm{ \signal-\nun_{\theta(\al)} \Bo_\al \Ao \signal
- \Bo_\al \Ao (\signal-\nun_{\theta(\al)} \Bo_\al \Ao \signal)}
\\& \qquad   +\|\Bo_\al \Ao \nun_{\theta(\al)}  \Bo_\al \Ao \signal\|
\\
 = & \|(\Id_X-\Bo_\al\Ao)(\signal-\nun_{\theta(\al)}  \Bo_\al \Ao \signal)\|
\\&
+  \|\Bo_\al \Ao \nun_{\theta(\al)}  \Bo_\al \Ao \signal\|\\
\leq \ &\|(\Id_X-\Bo_\al \Ao)(\Ao^\ast \Ao)^\mu \omega\|
\\ &+\|(\Id_X-\Bo_\al\Ao)(\signal -\nun_{\theta(\al)}  \Bo_\al \Ao
- (\Ao^\ast \Ao)^\mu \omega\|
\\ & \qquad+
\|\Bo_\al \Ao \nun_{\theta(\al)}  \Bo_\al \Ao \signal\|\\
\leq  & \|(\Id_X-\Bo_\al \Ao)(\Ao^\ast \Ao)^\mu \omega\|
\\& +d_\al(x; \rho, \mu)+\|\Bo_\al \Ao \nun_{\theta(\al)}  \Bo_\al \Ao \signal \| \,.
\end{align*}
Because $(\Bo_\al)_{\al>0}$ has qualification of order $\mu$, we have
\begin{equation*}
E_\al \leq C\rho\al^\mu+d_\al(x; \rho, \mu)  +\|\Bo_\al\Ao \nun_{\theta(\al)}  \Bo_\al \Ao \signal\| \,,
\end{equation*}
which concludes the proof.
\end{proof}

From Lemma \ref{lem:est} we obtain the following theorem providing
convergence rates for  families of RegNets.

\begin{theorem}[Convergence rate]\label{thm:rates}
Let $(\Ro_\al)_{\al >0}$ be a family of RegNets adapted to
$((\Bo_\al)_{\al>0}, \nullnet)$ for some classical regularization
$(\Bo_\al)_{\al}$ and $\M$ defined by a null space network 
$\Id_X + \nullnet$.
Further, assume that for a  set
$\M_{\rho,\mu} \subseteq \M$  the following hold:
\begin{enumerate}[leftmargin=2.7em,label=(A\arabic*)]
\item \label{a1} The parameter choice rule  satisfies  $\al \asymp \delta^{\frac{2}{2\mu+1}}$.
\item \label{a2} For all $x \in \M_{\rho,\mu}$  we have
$$d_\al(x; \rho, \mu) = \mathcal{O}(\al^\mu) \text{ as } \al \to 0$$
\item \label{a3} For all $x \in \M_{\rho,\mu}$ we have
$$\|\Bo_\al\Ao \nun_{\theta(\al)}  \Bo_\al \Ao \signal\|=\mathcal{O}(\al^\mu)
\text{ as } \al \to 0 \,.$$
\item \label{a4} $(\Bo_\al)_{\al >0}$ has qualification at least $\mu$.
\end{enumerate}
Then  for all $\signal \in \M_{\rho,\mu}$ the following convergence rates
 result holds
\begin{equation}
\| \Ro_\al(\ndata)-x\| = \mathcal{O}(\delta^{\frac{2\mu}{2\mu+1}})
\text{ as } \al \to 0 \,.
\end{equation}
\end{theorem}

\begin{proof}
The assertion follows from Lemma \ref{lem:est}.
\end{proof}

In the following section, we will give three examples of
regularization methods that arise as special  cases of
our  results given above. In  particular, we give a data driven
extension of  SVD regularization where the assumptions of
Theorem~\ref{thm:rates} are satisfied.

\section{Special cases}
\label{sec:spcases}

In this section, we demonstrate that our theory recovers known
existing results as special cases and demonstrate how to derive
novel  data driven regularization methods.
In particular, we show  that any classical regularization method,
regularization by null space networks and a deep learning variant
of truncated SVD fit within our  framework  introduced in
Section~\ref{sec:regprops}.

\subsection{Classical filter-based regularization}

Classical Tikhonov regularization is a special case of the regularization method defined in Theorem~\ref{theo:conv}
with
\begin{align*}
&\Bo_\al  =(\Ao^\ast\Ao+\al \Id_X)^{-1}\Ao^\ast \\
&\nun_{\theta(\al)}  =0 \,.
\end{align*}
In this case the distance function
\begin{equation*}
d_\al(x; \rho, \mu)=\inf\{\|x-(\Ao^\ast\Ao)^\mu \omega \| \mid \omega\in \X \wedge \|\omega\|\leq \rho\}
\end{equation*}
 is independent of $\al$ and therefore satisfies
 $d_\al(x; \rho, \mu) = \mathcal{O}(\al^\mu)$ if and only if $d_\al(x; \rho, \mu)=0$.
  This in turn is equivalent to
  \begin{equation*}
  x\in\{(\Ao^\ast \Ao)^\mu \omega \mid \omega\in \X \wedge \|\omega\|\leq \rho\}\,,
   \end{equation*}
   which is the classical source condition for the convergence
   rate $\|x-x_{\al,\delta}\|=\mathcal{O}(\delta^{\frac{2\mu}{2\mu+1}})$
   as $\delta \to 0$.

Clearly,  the above  considerations equally apply to any filter-based
regularization method including iterative Tikhonov regularization, truncated SVD,
and the Landweber iteration. We conclude that Theorem~\ref{thm:rates} contains classical
convergence rates results for classical regularization methods as special cases.

\subsection{Regularized null space networks}

In the case of  regularized null space networks, we
take $(\Bo_\al)_{\al>0}$ as a filter-based regularization method and
$\nun_{\theta(\al)}  = \nullnet$ for some null space network $\Id_X + \nullnet$.
In the following theorem we derive a decay rate of the distance function on the source set
\begin{equation*}
    \X_{\mu,\rho} \coloneqq \{(\Id_X+\nullnet)(\Ao^\ast \Ao)^\mu \omega \mid \omega \in \X 
    \text{ and } \|\omega\|\leq \rho\}
\end{equation*}
in the special case where the regularizing networks are given by a
regularized null space network.

For regularized null space networks, in \cite[Theorem 2.8]{schwab2018deep} we derive the convergence
rate  $\| \Ro_\al(\ndata)-x\| = \mathcal{O}(\delta^{\frac{2\mu}{2\mu+1}})$ for    $x \in \X_{\mu,\rho}$
and $\al \asymp \delta^{\frac{2}{2\mu+1}}$. The following theorem shows that
\cite[Theorem 2.8]{schwab2018deep} is  a special case of Theorem~\ref{thm:rates}.  In this sense, the results
of the current paper are indeed an extension of   \cite{schwab2018deep}.

\begin{theorem}[Convergence rates for regularized null space networks]
Let $\Id_X + \nullnet \colon \X \rightarrow \X$ be a null space network and take $\nun_{\theta(\al)} =\nullnet $ for all $\al >0$. Further, let $(\Bo_\al)_{\al>0}$ be a classical regularization method with qualification at least $\mu$
that satisfies $\nullnet \Bo_\al(0)=0$. Then  we have
\begin{equation}
d_\al(x; \rho, \mu)=\mathcal{O}(\al^\mu) \quad \text{for all $x \in \X_{\mu,\rho}$ }.
\end{equation}
In particular,  if $(\Bo_\al)_{\al >0}$ has qualification $\mu$ then the parameter choice $\al \asymp \delta^{2/(2\mu+1)} $ gives the convergence rate $\| \Ro_\al(\ndata)-x\| = \mathcal{O}(\delta^{2\mu/(2\mu+1)})$ for $x \in \X_{\rho, \mu}$.
\end{theorem}

\begin{proof}
For $x\in \X_{\mu,\rho}$ we have
\begin{align} \nonumber
\|x-&\nullnet\Bo_\al \Ao \signal -(\Ao^\ast \Ao)^\mu\omega\|
\\ \nonumber
=& \|\nullnet(\Ao^\ast \Ao)^\mu \omega - \nullnet\Bo_\al \Ao(\Ao^\ast \Ao)^\mu \omega
\\ \nonumber  & -\nullnet\Bo_\al \Ao\nullnet(\Ao^\ast \Ao)^\mu \omega\|
\\ \nonumber
=&\|\nullnet(\Ao^\ast \Ao)^\mu \omega - \nullnet \Bo_\al \Ao(\Ao^\ast \Ao)^\mu \omega\|
\\  \nonumber \leq & L  \|(\Id_X-\Bo_\al \Ao)(\Ao^\ast \Ao)^\mu \omega\|
\\ \nonumber  \leq & L C\al^\mu.
\end{align}
Here $L $ denotes the Lipschitz constant of $\nullnet$ and $C$ is some constant depending on the regularization $(\Bo_\al)_{\al>0}$.
\end{proof}

\subsection{Data-driven continued SVD}

For the following, assume that $\Ao$ admits a singular value decomposition
\begin{equation*}
\left( (u_n)_{n\in\N},(v_n)_{n\in\N}, (\sigma_n)_{n\in\N}\right) \,,
\end{equation*}
where $(u_n)_{n\in\N}$ and $(v_n)_{n\in\N}$ are orthonormal systems in $\X$ and $\Y$, respectively, and $\sigma_n$ are positive numbers such that for all $\signal \in \X$
\begin{equation}\label{eq:svd}
\Ao \signal = \sum_{n\in\N} \sigma_n \langle u_n,\signal\rangle v_n.
\end{equation}

The  regularization method corresponding to the regularizing filter given in
\eqref{eq:SVDfil} yields to the truncated SVD given by
\begin{equation}
\Bo_\al(\data) = \sum_{\sigma_n^2\geq \al} \frac{1}{\sigma_n}\langle \data, v_n \rangle u_n.
\end{equation}
The truncated SVD only recovers signal components corresponding to sufficiently large singular values of $\Ao$ and
sets the other components to zero. It seems reasonable to train a network that extends
the coefficients with nonzero values and therefore can better approximate
non-smooth  functions.

To achieve a learned data extension, we consider a family of regularizing networks of the form \eqref{eq:reg}
\begin{align}  \nonumber
\Ro_\al (\ndata)   &\coloneqq (\Id_X+\nun_{\theta(\al)} )\Bo_\al (\ndata)
\\ \label{eq:dc-SVD1}  & \qquad
=
(\Id_X+\nun_{\theta(\al)} ) \sum_{\sigma_n^2\geq \al} \frac{1}{\sigma_n}\langle \ndata, v_n \rangle u_n
\\  \nonumber
\nun_{\theta(\al)} (z)  &\coloneqq (\Id_X-\Bo_\al \Ao) \NN_{\theta(\al)}  (z)
 \\   \label{eq:dc-SVD2} & \qquad
= \sum_{\sigma_n^2 < \al} \langle \NN_{\theta(\al)} z, u_n \rangle u_n \,.
\end{align}
For the data-driven continued SVD \eqref{eq:dc-SVD1}, \eqref{eq:dc-SVD2} the following
convergence rates result holds.

\begin{theorem}[Convergence rates for data-driven continued SVD]\label{thm:svd}
Let $(\Ro_\al)_{\al >0}$ be defined by \eqref{eq:dc-SVD1}, \eqref{eq:dc-SVD2}
and adapted to $((\Bo_\al)_{\al>0}, \nullnet)$, where $(\Bo_\al)_{\al>0}$ is given by truncated SVD 
and $\M$ is defined by \eqref{eq:Mset} for some null space network $\Id_X + \nullnet$. Moreover, assume that $d_\al(x; \rho, \mu) = \mathcal{O}(\al^\mu) $
for  all  $x \in \M_{\rho,\mu}$  in some set
$\M_{\rho,\mu} \subseteq \M$.
Then, provided that $\al \asymp \delta^{\frac{2}{2\mu+1}}$,
for  all $\signal \in \M_{\rho,\mu}$ we have
\begin{equation}\label{eq:dc-SVD3}
\| \Ro_\al(\ndata)-x\| = \mathcal{O}(\delta^{\frac{2\mu}{2\mu+1}})
\text{ as } \al \to 0 .
\end{equation}
\end{theorem}

\begin{proof}
We apply Theorem \ref{thm:rates} and for that purpose  verify \ref{a1}-\ref{a4}.
Items \ref{a1} and \ref{a2} are satisfied according to the made assumptions.
Moreover, we have  $$\ran( (\Id_X-\Bo_\al \Ao) \NN_{\theta(\al)}  )\subseteq \operatorname{span}\{u_i \mid  \sigma_i^2<\al\} \,.
$$
Then for $\signal \in \X$ and all $\al$,  $\|\Bo_\al\Ao \nun_{\theta(\al)}  \Bo_\al \Ao \signal\|$ vanishes and therefore \ref{a3}  is satisfied. Finally, it is well known  that truncated SVD has infinite qualification \cite[Example 4.8]{engl1996regularization}, which gives Assumption \ref{a4} in Theorem \ref{thm:rates} and concludes the proof.
\end{proof}

The  networks $\nun_{\theta(\al)} $ map  the truncated SVD reconstruction
$\Bo_\al (\ndata)$   lying in the space spanned by the reliable basis elements (corresponding to sufficiently large  singular values of the operator $\Ao$)  to coefficients unreliably predicted by $\Ao$.
Hence, opposed to  truncated SVD, $\Ro_\al$
 is some form of  continued SVD, where the extension of the unreliable  coefficients is  learned from the reliable ones  in a  data driven manner.

Opposed to the two previous examples,  for the data driven continued SVD we don't have a simple and explicit characterization for the sets  $\M_{\rho,\mu}$ in Theorem \ref{thm:svd}. These sets crucially depend on the nature of the networks $\nun_{\theta(\al)}$, the used training data and training procedure. Investigating and characterizing these sets in particular situations will be subject of future research. 

Another natural example is the case where
classical Tikhonov regularization $
\Bo_\al  =(\Ao^\ast\Ao+\al \Id_X)^{-1}\Ao^\ast$ is
used to define a RegNet $(\Ro_\al)_\al $ of the 
form~\eqref{eq:reg}.  Also in this example, Theorem \ref{theo:conv}  gives convergence of $(\Ro_\al)_\al $ under the assumption that $(\nun_{\theta(\al)})_{\al >0}$ is adapted to $((\Bo_\al)_{\al>0}, \nullnet)$. However, for Tikhonov regularization we are currently not able to verify \ref{a3} under natural assumptions, required for the convergence rates results. Investigating convergence rates for the combination of Tikhonov regularization or other regularization methods with a learned component will be investigated in future research.

\section{Numerical Example}

In this section we consider the inverse problem $g = \radon(f)$, where
$\radon$ is  an undersampled  Radon transform.
For that purpose, we compare classical truncated SVD,
the data-driven extended SVD and the null-space approach of \cite{schwab2018deep}.
Similar results are presented in~\cite{schwab2019deep} for the limited data problem of
photoacoustic tomography.

\subsection{Discretization}

We discretize the Radon transform $\radon$ by using radial basis functions.
For a phantom $f\colon\R^2 \rightarrow \R$ supported in the domain $[-1,1]^2$  we make the basis function ansatz
\begin{equation}\label{eq:discrete_phantom}
    f(x) = \sum_{i=1}^{N^2} c_i \varphi_i(x),
\end{equation}
for coefficients $c_i\in\R$ and $\varphi_i(x) = \varphi(x-x_i)$, where $x_i$ are arranged on a Cartesian grid on $[-1,1]^2$ and $\varphi\colon\R^2\to\R$ is the Kaiser-Bessel function given by
\begin{equation}\label{eq:kaiserbessel}
    \varphi(x) = \begin{cases}
	\frac{I_0\left((\rho \sqrt{1-(\|x\|/a)^2)}\right)}{I_0(\rho)} \quad & \norm{x}\leq a\,,\\
	0 & \text{otherwise} \,.
    \end{cases}
\end{equation}
Here  $I_0$ denotes the modified first kind Bessel function and the parameters controlling the shape and support are chosen $\rho=7$ and $a=0.055$ (around 4 pixels in the images shown below), respectively.
We take advantage of the fact that for Kaiser-Bessel functions  the Radon-transform is known analytically~\cite{lewitt1990multidimensional}.

For our simulations we evaluate the Radon-transform at $N_\theta=30$ equidistant angles in $\theta_k\coloneqq {(k-1)\pi}/{N_\theta}$ and $N_s=200$ equidistant distances to the origin in the interval $[-3/2,3/2]$. Further, we use a total number of $N^2=128^2$ basis function to approximate the unknown density $f$.
Then the discrete forward operator $\Ao\in\R^{N_s N_\theta\times N^2}$ is defined by
$\Ao_{N_s(n-1)+j,i} = \radon(\varphi_i) (s_n,t_j)$.
This results in the following inverse problem for the coefficients of the phantom
\begin{equation}\label{eq:num_problem}
    \text{Recover } c\in\R^{N^2} \text{ from data}\;\; \yy = \Ao c+\xi.
\end{equation}
Here the vector $\xi\in \R^{N^2}$ models the error in the data.

For our choice of $N_\theta$, the Radon-transform is highly undersampled and \eqref{eq:num_problem} is ill-conditioned.
In the following we consider the problem of recovering $c$, since the function $f$ can be reconstructed by evaluating \eqref{eq:discrete_phantom}. Note that $\varphi_i$ are translated versions of a fixed basis function with centers on a Cartesian grid. Therefore, we can naturally arrange the coefficients $c\in\R^{N^2}$ as an $N\times N$ image.
This image representation will be used for visualization and for the inputs of the regularizing networks.

\begin{figure}[htb!]
    \includegraphics[width=\textwidth]{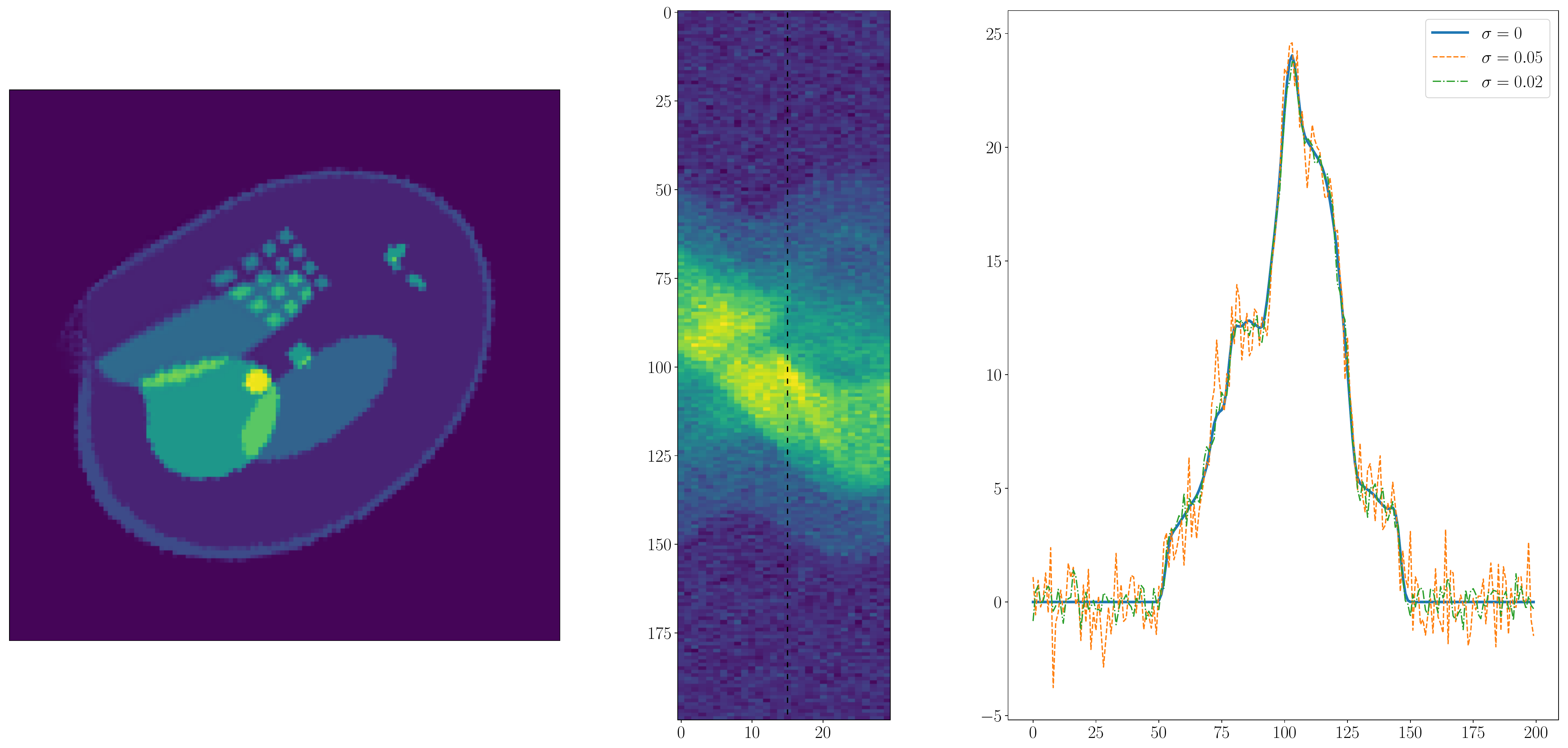}
    \caption{  \textsc{Right:} True phantom from the test set. \textsc{Middle:} Simulated sparse Radon data   $\Ao \xx+\delta\xi$ for $\N_\theta = 30$ directions, where $\xi_j\sim \norm{\Ao\xx}_\infty\mathcal{N}(0,1)$ with $\delta=0.05$.
    \textsc{Left:} Cross section of  the  data for the
    15th sensor directions for different noise levels.}
    \label{fig:true_data}
\end{figure}

\subsection{Used regularization methods}

Let $\Ao = U \Sigma V^\intercal$ be the singular value decomposition of the discrete forward operator. We denote by $(\uu_n)_{n=1}^{N_tN_\theta}$ and $(\vvv_n)_{n=1}^{N^2}$ the columns of $U$ and $U$ respectively and by $\sigma_1\geq\sigma_2\geq\ldots\geq\sigma_{N_sN_\theta}$ the singular values. Singular vectors  $\uu_n$ with vanishing singular values correspond  to components of the null space  $\ker (\Ao)$.

\begin{itemize}
\item
The truncated SVD $(\Bo_\al)_{\alpha>0}$ is then given by
\begin{equation}\label{eq:trunc_SVD}
    \Bo_\al(\yy) = \sum_{\sigma_n^2\geq\alpha} \frac{1}{\sigma_n}\innerprod{\yy,\vvv_n} \uu_n \quad \text{for } \yy\in\R^{N_sN_\theta}\,.
\end{equation}

\item The data-driven continued SVD (see  \eqref{eq:dc-SVD1}, \eqref{eq:dc-SVD2}) is  of the form
\begin{equation}\label{eq:network_exp}
    \Ro_\al(\yy) = \Bo_\al(\yy) + \sum_{\sigma_n^2<\alpha}\innerprod{\NN_{\theta(\al)}(\Bo_\al \yy),\uu_n}\uu_n\,,
\end{equation}
where $\NN_{\theta(\al)}\colon\R^{N^2}\to\R^{N^2}$ is a neural network
that operates on elements of $\R^{N^2}$ as $N\times N$ images, subsequently followed by the projection onto the singular vectors corresponding to the truncated singular values. We use the same U-net architecture as described in~\cite{PATsparse} (without residual connection) for $\NN_{\theta(\al)}$.
Note that the network does not affect the non-vanishing coefficients of the truncated SVD, which means that $\Ro_\al$ and $\Bo_\al$ reconstruct the same low frequency parts.

\item
Additionally, we apply the regularized null space network of \cite{schwab2018deep} which with the help of
the SVD can be evaluated by
\begin{equation}\label{eq:network_exp-n}
    \Ro_{\al}^0 (\yy) = \Bo_\al(\yy) + \sum_{\sigma_n^2 = 0}\innerprod{\NN_{\theta(\al)}^0(\Bo_\al \yy),\uu_n}\vvv_n \,.
\end{equation}
For the neural network $\NN_{\theta(\al)}$
we use again the U-net architecture as described
as above. Opposed  to \eqref{eq:network_exp}, the null space networks only add components of the kernel $\ker (\Ao)$
to $ \Bo_\al$.
\end{itemize}

Note that the implemented regularization  methods  fit in the general framework of RegNets, see 
Section~\ref{sec:spcases}. In particular, for all methods we have convergence as $\delta \to 0$. For the data driven continued SVD \eqref{eq:network_exp} this convergence result requires that there is some 
network $\NN \colon \X \to \X$ such that for all $\xx \in \ran (\Ao^\plus)$ we have   
\begin{equation*}
\lim_{\al \to 0}
\sum_{\sigma_n^2<\alpha}\innerprod{
\NN_{\theta(\al)} (\Po_\alpha \xx), \uu_n} \uu_n
=
\sum_{\sigma_n = 0 }\innerprod{\NN \xx,\uu_n}\uu_n \,,
\end{equation*}
where $ \Po_\alpha(\xx) \coloneqq \sum_{\sigma_n^2\geq\alpha} \innerprod{\xx,\uu_n} \uu_n$. We think that this convergence (at least on a reasonable 
subset of $\ran (\Ao^\plus)$) is reasonable using the same training strategy \eqref{eq:train_err} as below. Further theoretical and practical research, however, is required for rigorously analyzing this issue.

\subsection{Network training and reconstruction results}

The regularizing networks $\Ro_\al$ and $\Ro_\al^0$  were trained for different regularization parameters $\alpha$. Our training set consists of 1000 Shepp-Logan type phantoms  $c^{(k)}$ for
$k = 1, \dots ,1000$ as ground truth and the corresponding regularized reconstructions $\Bo_\al \yy^{(k)}$  where the data $\yy^{(k)}=\Ao c^{(k)}$ was simulated with the discrete forward operator $\Ao$.
We trained the network $\Ro_\al$ (and likewise $\Ro_\al^0$) by minimizing the mean absolute error (MAE)
\begin{equation}\label{eq:train_err}
\frac{1}{1000}\sum_{k=1}^{1000} \|c^{(k)}-\Ro_\al(\yy^{(k)})\|_1,
\end{equation}
with the stochastic gradient descent (SGD) algorithm. The learning rate was set to 0.05 and the momentum parameter to 0.99.
To evaluate the proposed regularizing networks we generated 250 phantoms for testing (see Figure~\ref{fig:true_data} for an example from the test set).

\begin{figure}[htb!]
    \includegraphics[width=\textwidth]{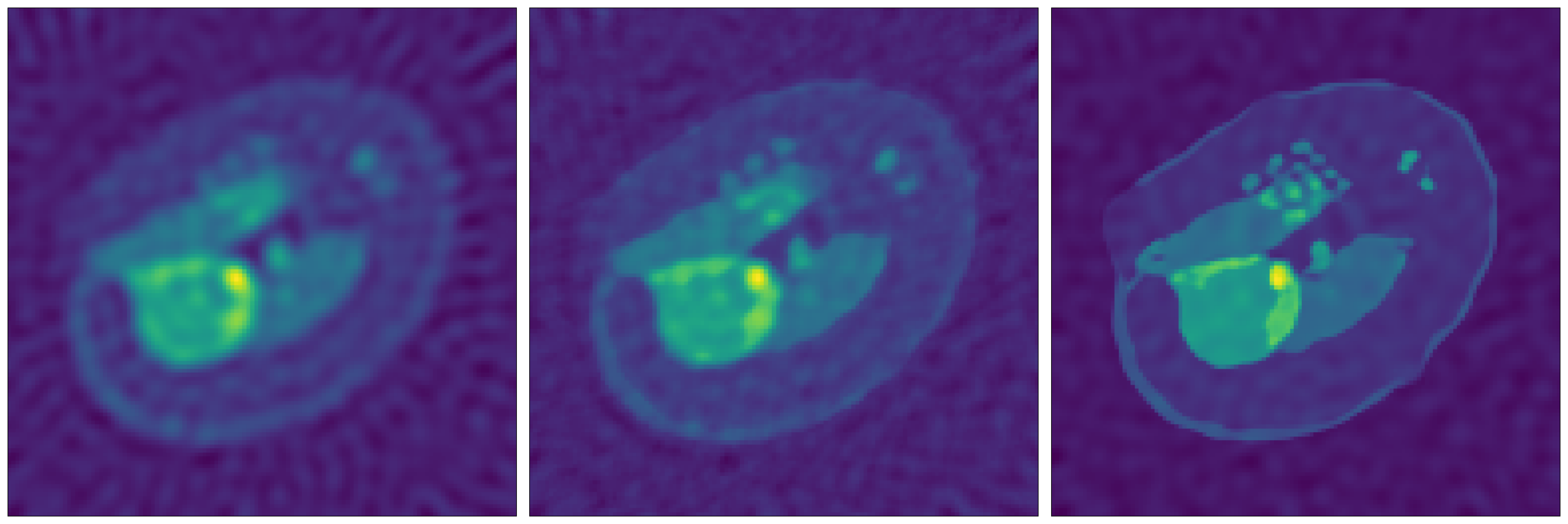}
    \caption{\textsc{Reconstructions} for low noise levels ($\delta=0.02$).
    \textsc{Left:} Truncated SVD. \textsc{Middle:} Nullspace network. \textsc{Right:} Reconstruction with continued SVD.
    }
    \label{fig:rec_lownoise}
\end{figure}

\begin{figure}[htb!]
    \includegraphics[width=\textwidth]{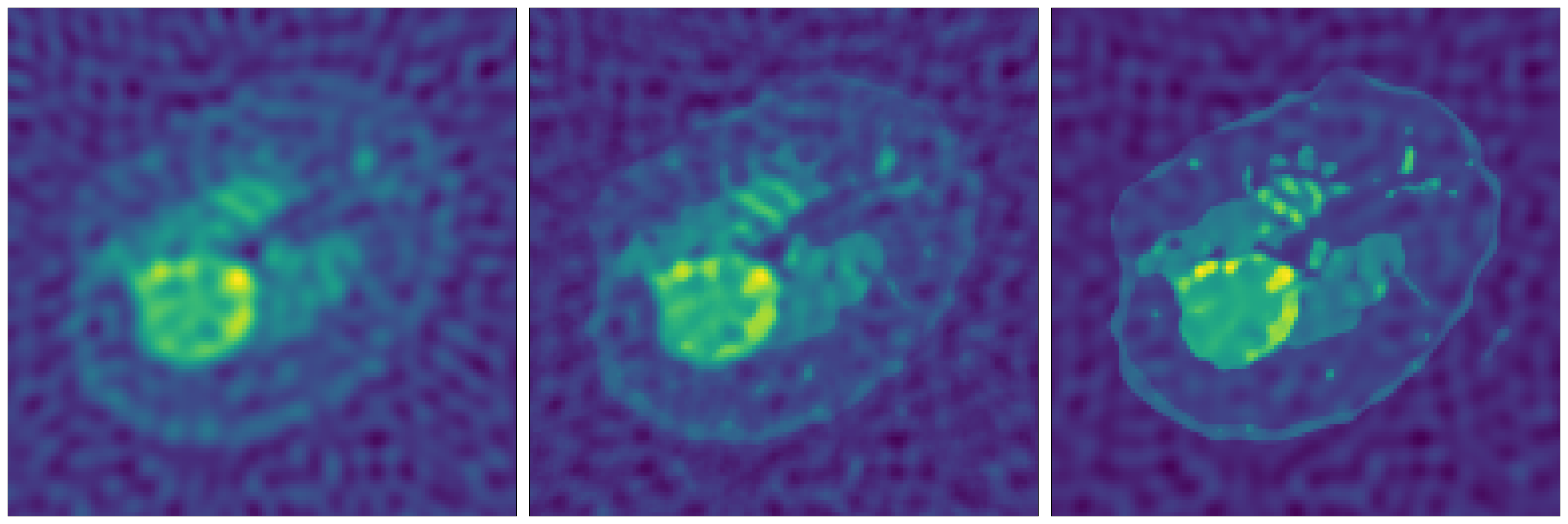}
    \caption{\textsc{Reconstructions} for higher noise levels ($\delta=0.05$).
    \textsc{Left:} Truncated SVD.
    \textsc{Middle:} Nullspace network.
    \textsc{Right:} Reconstruction with continued SVD.
    }
    \label{fig:rec_highnoise}
\end{figure}

We trained the networks $\Ro_\al$ and $\Ro_\al^0$ for 15 different values of the regularization parameter $\alpha$ the same way using noise free data minimizing \eqref{eq:train_err} for $\Ro_\al$ and $\mathbf{N}_\alpha$ respectively.   For the reconstructed images shown in Figure~\ref{fig:rec_lownoise} and Figure~\ref{fig:rec_highnoise} we took 10 different images with corresponding data  $\yy^{(k)} = \Ao\xx^{(k)} + \delta \xi^{(k)}$ with noise level of $\delta=0.05$, where $\xi^{(k)} \sim \snorm{\Ao\xx^{(k)}}_\infty \mathcal{N}(0,1)$.
Then we  chose the  regularization parameter with minimal   mean squared error, averaged over the 10 sample images. The resulting regularization parameter was $\alpha=1$ (which equals to taking the 796 biggest singular values).

For quantitative evaluation of the different approaches we calculated the mean errors for all 250 test images and all regularization parameters using the mean squared error (MSE) and the mean absolute error (MAE). All images were rescaled to have values in $[0,1]$ before calculating the error. The resulting error curves depending on the regularization parameter $\alpha$ (respectively, the number of used singular values)
are shown in  Figures  \ref{fig:errorplot_ln} and \ref{fig:errorplot_hn}.

\begin{figure}[htb!]
    \includegraphics[width=\textwidth]{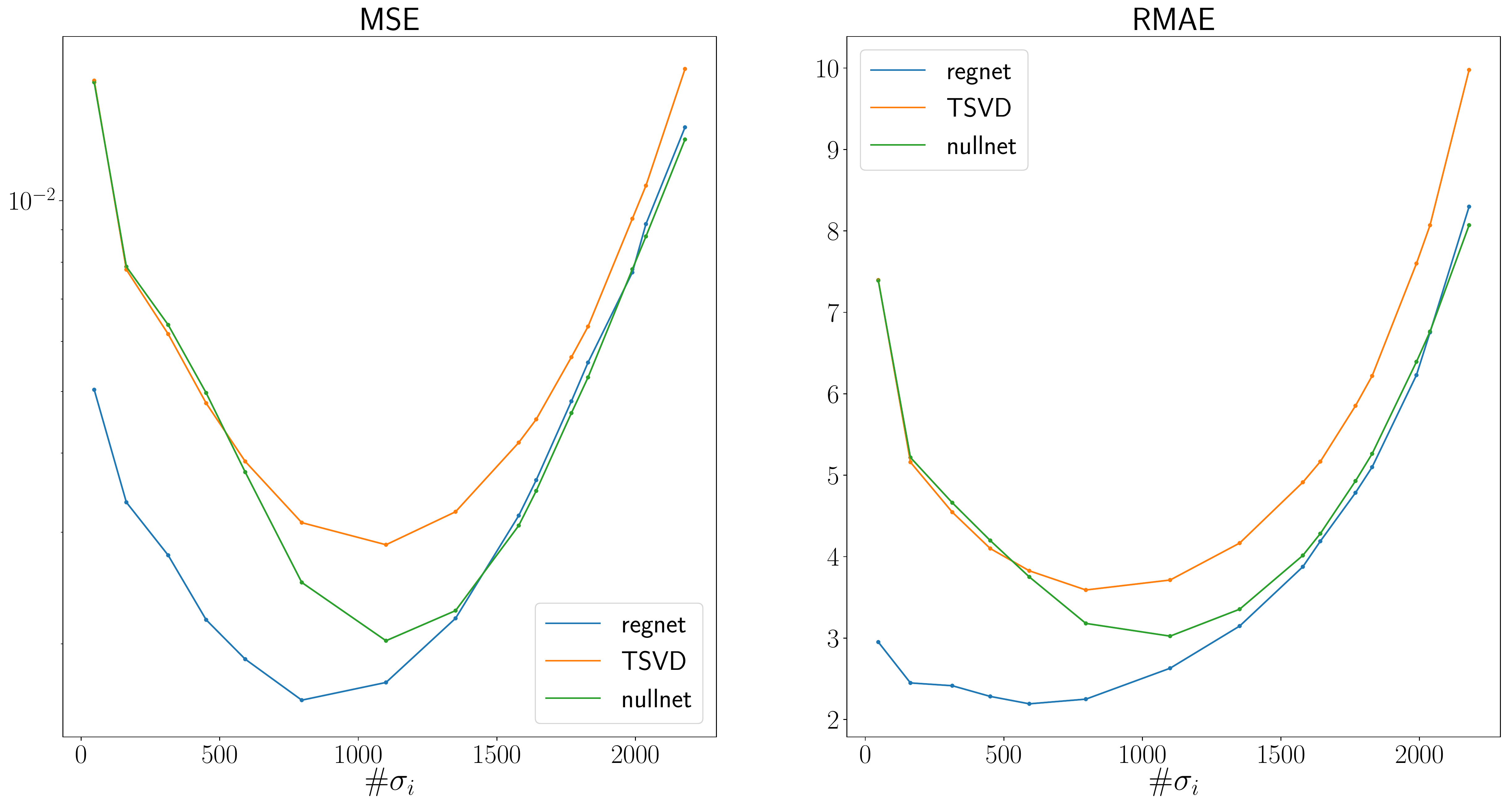}
    \caption{ \textsc{Mean Errors} for the test images using different error measures. On the $x$-axis are the number of used singular values. The noise level is $\delta=0.02$.
    }
    \label{fig:errorplot_ln}
\end{figure}

\begin{figure}[htb!]
    \includegraphics[width=\textwidth]{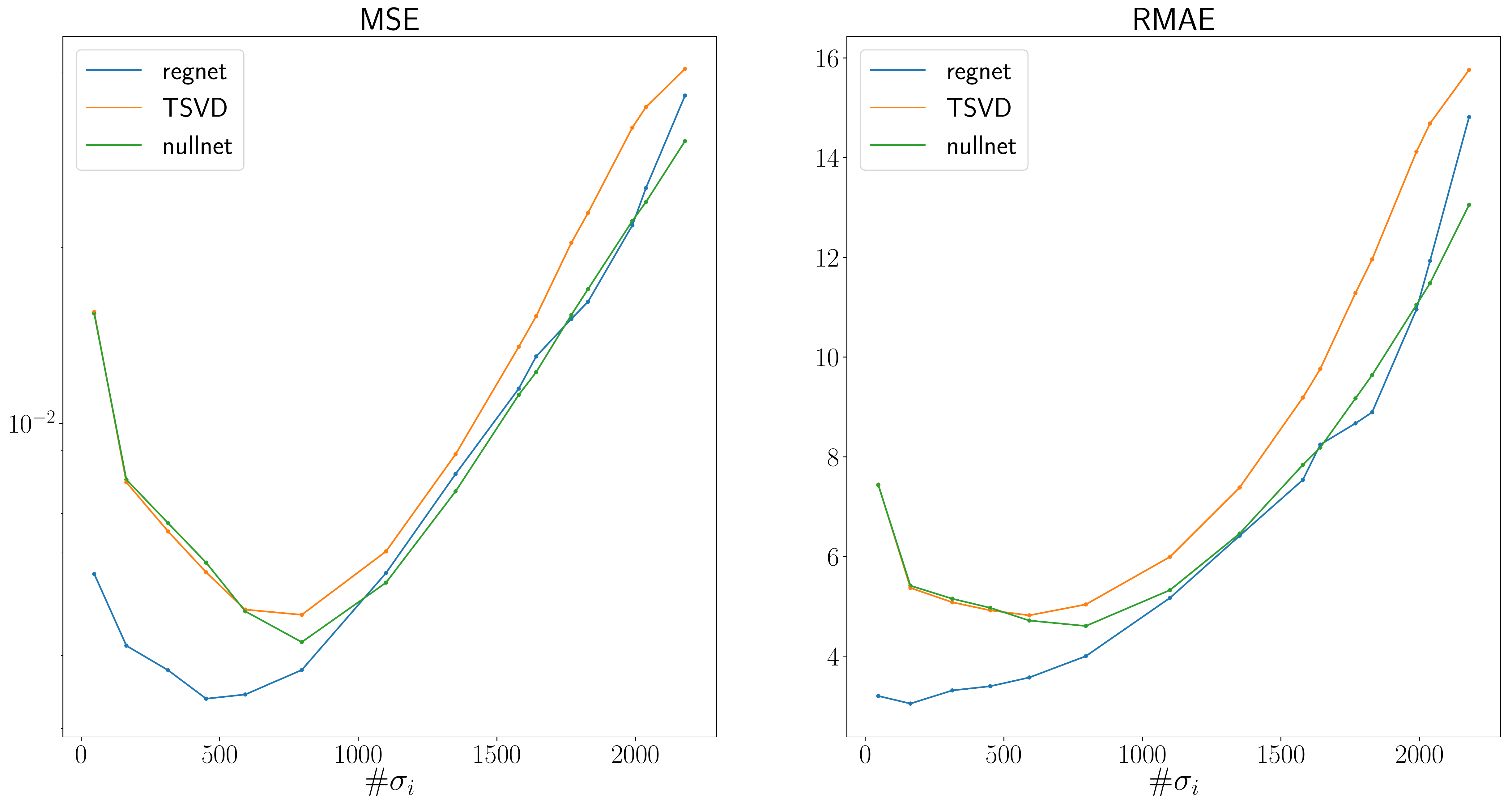}
    \caption{\textsc{Mean Errors} for the test images using different error measures. The $x$-axis shows  the number of used singular values. The noise level is  $\delta=0.05$.
    }
    \label{fig:errorplot_hn}
\end{figure}

\subsection{Discussion}

One can see that our proposed approach (data-driven continued
 SVD) in both cases outperforms the truncated SVD and the null space network; see  Figures~\ref{fig:rec_lownoise}  and~\ref{fig:rec_highnoise}. The better performance can also be  clearly seen in   Figures~\ref{fig:errorplot_ln} and \ref{fig:errorplot_hn}, where the reconstruction errors are shown for varying regularization parameter (the number of used singular values).
 The data-driven continued SVD  yields the smallest
 reconstruction errors followed by the null-space network and the
 truncated SVD.

 Interestingly, in these figures one also observes
 a shift to the left of the  error curve  for the methods with learned
 components compared to plain  truncated SVD.
 This  can be explained as follows. The continued SVD
 and the  null-space network preserve  the  singular components corresponding to large singular values. Further the reconstruction error corresponding to the truncated components is reduced  by applying the trained network  and therefore the overall error becomes reduced compared to the other two methods.
 We conclude that partially learned methods need less singular values to achieve  
accurate results.
 This effect is even larger for the learned SVD than for the null-space network.  This explains the improved performance of the learned SVD and the shift to the left in Figures~\ref{fig:errorplot_ln} and \ref{fig:errorplot_hn}.

There exists a variety of recently proposed deep learning based methods for solving inverse problems, and in
particular, for limited data problems in image reconstruction. Because the main  contribution of our work is the theoretical analysis we don't make the attempt here to numerically compare our method with other
deep learning based methods, for which no comparable theory is available.  
One  advantage  of our approach that we expect is the better generalization to data different from the training data. 
 Numerical studies investigating such issues is subject of future research.

\subsection{Extensions}

The probably most established deep learning approach to image reconstruction is to apply a two-step reconstruction network
$
\Ro_{\rm FBP}  
\coloneqq
(\Id + \NN_{\theta} ) \circ \Bo_{\rm FBP} $
where $\Bo_{\rm FBP}$ denotes the filtered backprojection operator and $(\Id + \NN_{\theta})$ is a trained residual network. The FBP $\Bo_{\rm FBP}$ can been seen as a regularization method in the case of full data. In the case of limited data this is not the case, and therefore it does not fully  fit into the framework of our theory.  Analyzing such more general situations opens  an interesting line of research,
that we aim to address in future work.

Another interesting generalization of our results is the extension to regularization also from left and from the right. In  this case the reconstruction networks have the form
\begin{equation*}
\Ro_{\al, \beta} (\data )
\coloneqq
\Bo_{\beta }^{(1)} (\Id + \nun_{\theta(\alpha, \beta)} ) \circ \Bo_{\al }^{(0)} \circ (\data) \,,
\end{equation*}
for regularization methods  $(\Bo_{\al}^{(0)})_\al$, $(\Bo_{\beta }^{(1)})_\beta$ and networks $\nun_{\theta(\alpha, \beta)} $. Extensions are even possible using cascades of network, which would have  similarity with  iterative and variational networks \cite{adler2017solving,kobler2017variational}  and cascades of networks \cite{kofler2018u,schlemper2017deep}. We expect that our results can be extended to such more general situations.

\section{Conclusion}
\label{sec:conclusion}

In this paper we introduced the concept of regularizing families of networks (RegNets), which are  sequences of deep CNNs. 
The trained components of the  networks, as well as the classical parts, are allowed to depend on the regularization parameter and it is shown, that under certain assumptions this approach yields a convergent regularization method. We also derived convergence rates under the assumption, that the solution lies in a source set, that is different from the classical source sets. Examples were given, where the assumptions are satisfied. It has been  shown, that the new
framework recovers results  for classical regularization as special cases as well as data driven improvements of classical regularization.
Such data driven regularization methods can give better  results in  practice than classical regularization methods which only use  hand crafted prior information.

As a numerical example, we investigated a sparse sampling  problem for  the Radon transform.  As  regularization method we took the truncated SVD and its data driven counterparts, the null-space network and the continued  SVD. Numerical results clearly demonstrate that   the continued  SVD outperforms classical SVD as well as the null space network.
Future work will be done to test the  proposed regularizing networks on  further ill-posed inverse problems and compare it with various  other regularization methods. A detailed numerical
comparison of our method with other deep learning  methods is subject of future research. This will reveal the theoretical advantage of  our method, that it actually  has improved generalizability.

\section*{Acknowledgement}
The work of M.H and S.A. has been supported by the Austrian Science Fund (FWF),
project P 30747-N32. Essential parts of this work have been finished during
the IUS conference 2018, October 22-25,  in Japan.


\begin{thebibliography}{10}
\providecommand{\url}[1]{{#1}}
\providecommand{\urlprefix}{URL }
\expandafter\ifx\csname urlstyle\endcsname\relax
  \providecommand{\doi}[1]{DOI~\discretionary{}{}{}#1}\else
  \providecommand{\doi}{DOI~\discretionary{}{}{}\begingroup
  \urlstyle{rm}\Url}\fi

\bibitem{adler2018banach}
Adler, J., Lunz, S.: Banach {W}asserstein {GAN}.
\newblock In: Advances in Neural Information Processing Systems, pp. 6754--6763
  (2018)

\bibitem{adler2017solving}
Adler, J., {\"O}ktem, O.: Solving ill-posed inverse problems using iterative
  deep neural networks.
\newblock Inverse Probl. \textbf{33}(12), 124007 (2017)

\bibitem{adler2018deep}
Adler, J., {\"O}ktem, O.: Deep {B}ayesian inversion.
\newblock arXiv:1811.05910  (2018)

\bibitem{adler2017learning}
Adler, J., Ringh, A., {\"O}ktem, O., Karlsson, J.: Learning to solve inverse
  problems using {W}asserstein loss.
\newblock arXiv:1710.10898  (2017)

\bibitem{PATsparse}
Antholzer, S., Haltmeier, M., Schwab, J.: Deep learning for photoacoustic
  tomography from sparse data.
\newblock Inverse Problems in Science and Engineering \textbf{27}(7), 987--1005
  (2019).
\newblock \doi{10.1080/17415977.2018.1518444}

\bibitem{bubba2019learning}
Bubba, T.A., Kutyniok, G., Lassas, M., M{\"a}rz, M., Samek, W., Siltanen, S.,
  Srinivasan, V.: Learning the invisible: A hybrid deep learning-shearlet
  framework for limited angle computed tomography.
\newblock Inverse Probl. \textbf{35}(6), 064002 (2019)

\bibitem{engl1996regularization}
Engl, H.W., Hanke, M., Neubauer, A.: Regularization of inverse problems, vol.
  375.
\newblock Springer Science \& Business Media (1996)

\bibitem{gupta2018cnn}
Gupta, H., Jin, K.H., Nguyen, H.Q., McCann, M.T., Unser, M.: {CNN}-based
  projected gradient descent for consistent ct image reconstruction.
\newblock IEEE Trans. Med. Imag. \textbf{37}(6), 1440--1453 (2018)

\bibitem{han2018framing}
Han, Y., Ye, J.C.: Framing {U}-net via deep convolutional framelets:
  Application to sparse-view {CT}.
\newblock IEEE Trans. Med. Imag. \textbf{37}(6), 1418--1429 (2018)

\bibitem{Hel69}
Helmberg, G.: Introduction to spectral theory in {H}ilbert space.
\newblock North-Holland Series in Applied Mathematics and Mechanics, Vol. 6.
  North-Holland Publishing Co., Amsterdam (1969)

\bibitem{hofmann2005convergence}
Hofmann, B., Yamamoto, M.: Convergence rates for {T}ikhonov regularization
  based on range inclusions.
\newblock Inverse Probl. \textbf{21}(3), 805 (2005)

\bibitem{kobler2017variational}
Kobler, E., Klatzer, T., Hammernik, K., Pock, T.: Variational networks:
  connecting variational methods and deep learning.
\newblock In: German Conference on Pattern Recognition, pp. 281--293. Springer
  (2017)

\bibitem{kofler2018u}
Kofler, A., Haltmeier, M., Kolbitsch, C., Kachelrie{\ss}, M., Dewey, M.: A
  {U}-nets cascade for sparse view computed tomography.
\newblock In: International Workshop on Machine Learning for Medical Image
  Reconstruction, pp. 91--99. Springer (2018)

\bibitem{lewitt1990multidimensional}
Lewitt, R.M.: Multidimensional digital image representations using generalized
  {K}aiser--{B}essel window functions.
\newblock JOSA A \textbf{7}(10), 1834--1846 (1990)

\bibitem{li2018nett}
Li, H., Schwab, J., Antholzer, S., Haltmeier, M.: {NETT}: Solving inverse
  problems with deep neural networks.
\newblock arXiv:1803.00092  (2018)

\bibitem{mardani2017deep}
Mardani, M., Gong, E., Cheng, J.Y., Vasanawala, S., Zaharchuk, G., Alley, M.,
  Thakur, N., Han, S., Dally, W., Pauly, J.M., et~al.: Deep generative
  adversarial networks for compressed sensing automates {MRI}.
\newblock arXiv:1706.00051  (2017)

\bibitem{mardani2017recurrent}
Mardani, M., Monajemi, H., Papyan, V., Vasanawala, S., Donoho, D., Pauly, J.:
  Recurrent generative adversarial networks for proximal learning and automated
  compressive image recovery.
\newblock arXiv:1711.10046  (2017)

\bibitem{ronneberger2015unet}
Ronneberge, O., Fischer, P., Brox, T.: U-net: Convolutional networks for
  biomedical image segmentation.
\newblock In: International Conference on Medical Image Computing and
  Computer-Assisted Intervention, pp. 234--241 (2015)

\bibitem{scherzer2009variational}
Scherzer, O., Grasmair, M., Grossauer, H., Haltmeier, M., Lenzen, F.:
  Variational methods in imaging, \emph{Applied Mathematical Sciences}, vol.
  167.
\newblock Springer, New York (2009)

\bibitem{schlemper2017deep}
Schlemper, J., Caballero, J., Hajnal, J.V., Price, A., Rueckert, D.: A deep
  cascade of convolutional neural networks for mr image reconstruction.
\newblock In: International Conference on Information Processing in Medical
  Imaging, pp. 647--658. Springer (2017)

\bibitem{schwab2018deep}
Schwab, J., Antholzer, S., Haltmeier, M.: Deep null space learning for inverse
  problems: convergence analysis and rates.
\newblock Inverse Probl. \textbf{35}(2), 025008 (2019)

\bibitem{schwab2019deep}
Schwab, J., Antholzer, S., Nuster, R., Paltauf, G., Haltmeier, M.: Deep
  learning of truncated singular values for limited view photoacoustic
  tomography.
\newblock In: Photons Plus Ultrasound: Imaging and Sensing 2019, vol. 10878, p.
  1087836 (2019)

\bibitem{Wei80}
Weidmann, J.: Linear Operators in Hilbert Spaces, \emph{Graduate Texts in
  Mathematics}, vol.~68.
\newblock Springer, New York (1980)

\bibitem{ye2018deep}
Ye, J.C., Han, Y., Cha, E.: Deep convolutional framelets: A general deep
  learning framework for inverse problems.
\newblock SIAM J. Imaging Sci. \textbf{11}(2), 991--1048 (2018)

\end{thebibliography}
\end{document}